\documentclass [a4paper,12pt]{amsart}
\usepackage{amsthm}
\usepackage{amsmath}
\usepackage{amssymb}
\usepackage[ansinew]{inputenc}
\usepackage{amscd}
\usepackage[ansinew]{inputenc}
\usepackage[mathscr]{euscript}
\usepackage{color}
\usepackage{enumerate}
\usepackage{enumitem}

\usepackage{hyperref}

\newtheorem{thm}{Theorem}[section]
\newtheorem{cor}[thm]{Corollary}

\newtheorem{prop}[thm]{Proposition}

\theoremstyle{definition}
\newtheorem{ex}[thm]{Example}

\theoremstyle{definition}

\theoremstyle{definition}
\newtheorem{rem}[thm]{Remark}

\theoremstyle{definition}

\def\C{\mathbb C}
\def\R{\mathbb R}

\def\Z{\mathbb Z}

\def\supp {\mathrm{supp}}

\def\ord{\operatorname{ord}}
\def\O{\mathcal O}

\def\m{\mathbf m}

\def\LL{\mathcal L}
\def\leq{\leqslant}
\def\geq{\geqslant}

\def\({\left(}
\def\){\right)}

\def\to{\longrightarrow}
\def\DP{\operatorname{DP}}
\def\icis{\textsc{icis}}
\def\Li{B}
\def\lct{\operatorname{lct}}
\def\*{\color{red}\blacksquare}

\newcommand{\tL}{\textnormal{\texttt{L}}}

\subjclass[$2000$ Mathematics Subject Classification]{Primary
32S05; Secondary 32S30}

\begin{document}

\title[The log canonical threshold of products of ideals]
{The log canonical threshold of products of ideals \\ \vskip4pt and mixed \L ojasiewicz exponents}

\author{Carles Bivià-Ausina}

\address{Institut Universitari de Matemàtica Pura i Aplicada, Universitat Politècnica de València,
Camí de Vera s/n, 46022 València, Spain}

\email{carbivia@mat.upv.es}

\thanks{Work partially supported by Grant PID2021-124577NB-I00 funded by
Ministerio de Ciencia, Innovación y Universidades MCIN/AEI/10.13039/501100011033 and by \lq ERDF A way of making Europe' (European Regional Development Fund)}

\keywords{\L ojasiewicz exponents, log canonical threshold, integral closure of ideals.}

\dedicatory{\rightline{To the memory of Professor Arkadiusz P\l oski}}

\begin{abstract}
Given two ideals $I$ and $J$ of the ring $\O_n$ of analytic function germs $f:(\C^n,0)\to \C$,
we show a sharp lower bound for the log canonical threshold of $IJ$ in terms of the sequences of mixed \L ojasiewicz exponents of them.
In particular, in the case where $J$ is the maximal ideal, the corresponding equality holds if and only if
the integral closure of $I$ equals some power of the maximal ideal.
\end{abstract}

\maketitle


\section{Introduction}\label{Intro}

Let $\O_n$ denote the ring of holomorphic function germs $f:(\C^n,0)\to \C$ and let us denote by $\m_n$ the maximal ideal of $\O_n$. Given a germ $f\in\O_n$ with isolated singularity at the origin, let $\mu^*(f)$ be the vector of mixed Milnor numbers
$(\mu^{(n)}(f),\dots, \mu^{(1)}(f), \mu^{(0)}(f))$, introduced by Teissier in \cite{Cargese}. We recall that $\mu^{(i)}(f)$ denotes the Milnor number of the restriction $f\vert_H$ of $f$ to a generic linear subspace $H$ in $\C^n$ of dimension $i$, for all $i=0,1,\dots, n$. Therefore $\mu^{(n)}(f)$ is the Milnor number of $f$, which is commonly denoted by $\mu(f)$. Moreover $\mu^{(1)}(f)=\ord(f)-1$ and $\mu^{(0)}(f)=1$, where $\ord(f)$ denotes the order of $f$, that is, the maximum of those $r\in\Z_{\geq 1}$ such that $f\in\m_n^r$.
The study of mixed Milnor numbers of function germs of $\O_n$, including the relations between them, motivated the introduction in \cite{Cargese} of the sequence of mixed multiplicities attached to any pair of ideals of $\O_n$ of finite colength. This idea was further explored by Rees \cite{Rees2}, thus leading to an essential research line in commutative algebra. In particular, Rees introduced the notion of mixed multiplicity $e(I_1,\dots, I_n)$
of any given collection $I_1,\dots, I_n$ of $\m$-primary ideals in a Noetherian local ring of dimension $n$ and
gave results about the computation of this number via the notion of joint reduction of ideals. Formulas relating the multiplicity of a product of ideals with the family of mixed multiplicities attached to them are also given in \cite{Rees2,Cargese} and have been of fundamental interest
in the study of the constancy of several invariants attached to a given analytic deformation $f_t:(\C^n,0)\to \C$ of fuction germs (we refer to
\cite{G92, G96} for subsequent generalizations in the context of isolated complete intersection singularities).

In addition, if we fix an analytic function germ $f:(\C^n,0)\to (\C,0)$ and any integer $i\in\{1,\dots, n\}$, Teissier studied in \cite[\S 5]{TeissierIM} and \cite{TeissierRIMS} the \L ojasiewicz exponent of the Jacobian ideal of the restriction $f\vert_H$ of $f$ to a generic linear subspace $H$ in $\C^n$ of dimension $i$. Similarly to \cite{TeissierRIMS}, let us denote by $\theta^{(i)}(f)$, for $i\in\{1,\dots, n\}$, the
family of numbers thus obtained. Teissier conjectured in \cite[p.\,7]{TeissierRIMS} that
$$
\sum_{i=1}^n\frac{1}{\theta^{(i)}(f)+1}\leq \sigma_0(f)
$$
where $\sigma_0(f)$ denotes the Arnold exponent of $f$ (this number is also called in \cite[p.\,273]{Kollar} the {\it complex singular index} of $f$). This conjecture has been proved by Dirks and Musta\c{t}\u a in \cite{DM} in a more general context. We also refer to \cite{EM,Loeser}
for previous results in this direction. Let us recall that, if $\lct(f)$ denotes the log canonical threshold of $f$, then
$\lct(f)=\min\{1,\sigma_0(f)\}$, as can be see in \cite[Theorem 9.5]{Kollar}. Kim proved in \cite{Kim} the inequality
\begin{equation}\label{Kiminicial}
\sum_{i=1}^n\frac{1}{\theta^{(i)}(f)+1}\leq \lct(\m_nJ(f)),
\end{equation}
where we denote by $\lct(I)$ the log canonical threshold of a given ideal $I$ of $\O_n$ (see \cite{Mu2, Kollar}).
Inequality (\ref{Kiminicial}) has been our main motivation in this paper.

If $I$ denotes an ideal of $\O_n$ of finite colength, Hickel introduced in \cite{Hickel2} the sequence of
\L ojasiewicz exponents of $I$ that appear by restricting a given generating system of $I$ to generic linear subspaces of different dimensions.
Additionally, we introduced in \cite{Bivia2009} the notion of mixed \L ojasiewicz exponent $\LL_0(I_1,\dots, I_n)$
of a collection of $d$ ideals $I_1,\dots, I_n$ in a given Noetherian local ring of dimension $n$ under the condition that $\sigma(I_1,\dots, I_n)<\infty$, where $\sigma(I_1,\dots, I_n)$ denotes what we call the {\it Rees' mixed multiplicity of $I_1,\dots, I_n$}
(see (\ref{lasigma})). Hence, if $I$ is an ideal of $\O_n$ of finite
colength and $i\in\{1,\dots, n\}$, we can set $\LL_0^{(i)}(I)=\LL_0(I,\dots, I, \m_n,\dots, \m_n)$, where $I$ is repeated $i$ times and $\m_n$
is repeated $n-i$ times. As can be seen in \cite[Lemma 3.9]{BF1}, the sequence of numbers $\LL_0^{(n)}(I)\geq \cdots \geq \LL_0^{(1)}(I)$ thus obtained coincides with the above-mentioned sequence of numbers defined by Hickel in \cite{Hickel2}.
In this article we extend (\ref{Kiminicial}) by replacing the Jacobian ideal by an arbitrary ideal $I$ of $\O_n$ of finite colength and characterize the corresponding equality. That is, in Theorem \ref{novafita2} we prove that
\begin{equation}\label{introddesKim2}
\sum_{i=1}^n\frac{1}{\LL_0^{(i)}(I)+1} \leq \lct\(\m_nI\).
\end{equation}
and equality holds if and only if there exists some $r\in\Z_{\geq 1}$ such that $\overline I=\m_n^r$.
This result will be supported by Theorem \ref{novafita}, which is the other main result of the article, where we give a chain of inequalities relating $\lct(IJ)$, for any pair of ideals $I$ and $J$ of $\O_n$ of finite colength, with the sequences of mixed \L ojasiewicz exponents $\LL_0^{(i)}(I)$ and $\LL_0^{(i)}(J)$, for $i=1,\dots, n$. At the end of the article we analyze how inequality (\ref{Kiminicial}) behaves when applied to the
Briançon-Speder example $f_t:(\C^3,0)\to (\C,0)$ (see \cite{BSexample}).

Let us remark that the kind of results that we have obtained takes part of the more general project of determining
the connections between \L ojasiewicz exponents and log canonical thresholds of ideals.
Let us fix a coordinate system $x_1,\dots, x_n$ in $\C^n$ and let $I$ be an ideal of $\O_n$. Let $\Gamma_+(I)$ denote the Newton polyhedron of $I$
with respect to the coordinates $x_1,\dots, x_n$ (see \cite{BFS} for instance). We say that $I$ is {\it monomial} when $I$
admits a generating system formed by monomials in $x_1,\dots, x_n$.
Let us recall that if $I$ is monomial, then the log canonical threshold of $I$ is expressed by means of the Newton polyhedron of $I$ (see (\ref{mainHowald})), by a result of Howald \cite{Howald}. As can be seen in \cite[Proposition 5.3]{BF2}, if $\overline I$ is monomial and $I$ has finite colength, then $\Gamma_+(I)$ also determines all rational powers of $I$ and therefore it gives any \L ojasiewicz exponent $\LL_J(I)$, where $J$ denotes
any other proper ideal of $\O_n$ (see Section \ref{preliminary} for the definition of $\LL_J(I)$ and other preliminary concepts and results). Both results allowed us to find that $\lct(I)=\LL_{x_1\cdots x_n}(I)$ when $\overline I$ is a
monomial ideal of $\O_n$ (see \cite[Theorem 5.4]{BF2}). In view of these facts and the existence of a deformation of $I$ into
its initial ideal (after fixing a multiplicative order of the coordinates $x_1,\dots, x_n$) by a flat family of ideals (see \cite[Corollary 7.5.2]{GP}), one could expect that, in general log canonical thresholds and \L ojasiewicz exponents
have the same nature, in other words, have very similar properties. However, semicontinuity in deformations is an essential feature that differentiates them.

Log canonical thresholds of ideals are lower semicontinuous in deformations of ideals, as can be seen in \cite{DK} or \cite[p.\,418]{Mu2}. Moreover, if $f_t:(\C^n,0)\to (\C,0)$ is a $\mu$-constant deformation of analytic function germs, the numbers $\lct(f_t)$ remain constant. As a counterpart, the problem of deciding the constancy of $\LL_0(J(f_t))$ in Milnor constant deformations is still an open problem. By the results of Teissier \cite{TeissierIM} (see also \cite{TeissierRIMS}) it follows that the \L ojasiewicz exponent of the ideal $J(f_t)$ is lower semicontinuous if the deformation $f_t$ is $\mu$-constant. This result was generalized by A. P\l oski in \cite{Ploski2010} to deformations of map germs $(\C^n,0)\to (\C^n,0)$ with constant multiplicity. We also refer to \cite{MN2005,RodakKodai,RRS2016} for further works related with the
semicontinuity of \L ojasiewicz exponents in deformations of ideals.

\section{Mixed \L ojasiewicz exponents}\label{preliminary}   

Given two proper ideals $I$ and $J$ of $\O_n$ with $V(I)\subseteq V(J)$, if $f:(\C^n,0)\to (\C^p,0)$ and $g:(\C^n,0)\to (\C^q,0)$ are
analytic maps whose component functions generate $I$ and $J$, respectively, then the {\it \L ojasiewicz exponent of $I$ with respect to $J$}, denoted by
$\LL_J(I)$, is defined as as the infimum of those $\alpha\in\R_{\geq 0}$ for which there exists a constant $C>0$ and an open
neighbourhood $U$ of $0$ in $\C^n$ such that
$$
\Vert g(x)\Vert^\alpha \leq C\Vert f(x) \Vert
$$
for all $x\in U$. We will denote $\LL_{\m_n}(I)$ simply by $\LL_0(I)$ and we call this number the {\it \L ojasiewicz exponent of $I$}.
The numbers $\LL_J(I)$ have several characterizations, as can be seen in \cite[Théorème 7.2]{LT1974}.
We refer to \cite{BE2009,BE2011,BE2013,Brz,Feehan1,GKP,KOP,Ploski1984,Ploski1985,PloskiBCP,RS2011,Sp2000,TeissierResonances} for additional information regarding geometrical, topological and algebraic aspects of the numbers thus defined.

Let us denote by $\overline I$ the integral closure of $I$. By \cite{LT1974}, we also have that
\begin{equation}\label{rsint}
\LL_J(I)=\min\left\{\frac rs: J^r\subseteq \overline{I^s},\,r,s\in\Z_{\geq 1}\right\}.
\end{equation}
Moreover $\LL_0(I)$ is equal to the number $\tau^*(I)$ defined by
D'Angelo in \cite[p.\ 621]{DAngelo}. That is
\begin{equation}\label{ordre}
\LL_0(I)=\sup_{\gamma\in\mathcal P}\bigg(\inf_{h\in
I}\frac{\ord(h\circ\gamma)}{\ord(\gamma)}\bigg),
\end{equation}
where $\mathcal P$ denotes the set of analytic maps $(\C,0)\to
(\C^n,0)$. The number on the right of (\ref{ordre}) is called in \cite{DAngelo, DAngelo2}
the {\it order of contact of $I$}.

Given a Noetherian local ring $(R,\m)$, if $I$ and $J$ is a pair of ideals of $R$ such that $\sqrt J\subseteq \sqrt I$,
the expression of $\LL_J(I)$ given in (\ref{rsint}) in terms of the integral closures leads to the definition of
{\it \L ojasiewicz exponent of $I$ with respect to $J$} also for pairs of proper ideals $I$ and $J$ in arbitrary Noetherian local rings such that $\sqrt J\subseteq \sqrt I$.
If additionally, we assume that $I$ is $\m$-primary and $R$ is quasi-unmixed (see \cite[p.\,401]{HS} or
\cite[p.\,251]{Matsumura}), then the Rees' multiplicity theorem (see \cite[p.\,222]{HS}) allows to write
$$
\LL_J(I)=\min\left\{\frac rs: e(I^s)=e(J^r+I^s),\,r,s\in\Z_{\geq 1}\right\}.
$$
where $e(I)$ denotes the Samuel multiplicity of $I$ (see \cite{HS, Matsumura,V}).


Let $(R,\m)$ be a Noetherian local ring of dimension $n$. If $I_1,\dots, I_n$ is a collection of $n$ $\m$-primary ideals of $R$, then
we denote by $e(I_1,\dots, I_n)$ the {\it mixed multiplicity of $I_1,\dots, I_n$} (see \cite{HS, Rees2}). Let us recall that when $I_1,\dots, I_n$ coincide with a given ideal $I$ of $R$, then $e(I_1,\dots, I_n)=e(I)$. Given an integer
$i\in\{0,1,\dots, n\}$, we define $e_i(I)$ as the mixed multiplicity $e(I,\dots, I,\m,\dots, \m)$, where $I$ is repeated $i$ times
and $\m$ is repeated $n-i$ times. Hence we have $e_n(I)=e(I)$, $e_1(I)=\ord(I)$ and $e_0(I)=e(\m)$, where $\ord(I)$ denotes the order of the ideal, that is, the maximum of those $r\in\Z_{\geq 1}$ such that $I\subseteq \m^r$.

More generally, in \cite{BiviaMRL} we introduced the number $\sigma(I_1,\dots, I_n)$, where now all the ideals are not assumed to
be $\m$-primary:
\begin{equation}\label{lasigma}
\sigma(I_1,\dots, I_n)=\max_{r\in\Z_{\geq 1}}\,e\big(I_1+\m^r,\dots,
I_n+\m^r\big).
\end{equation}

As a consequence of \cite[Proposition 2.9]{BiviaMRL}, if the residual field $R/\m$ is infinite,
the condition $\sigma(I_1,\dots, I_n)<\infty$ is equivalent to the existence of elements $g_i\in I_i$, for all $i=1,\dots, n$, such that
the ideal $\langle g_1,\dots, g_n\rangle$ is $\m$-primary. So, if $\sigma(I_1,\dots, I_n<\infty$, then $\sigma(I_1,\dots, I_n)$ is the lowest possible value of the multiplicity $e(g_1,\dots, g_n)$, where $(g_1,\dots, g_n)$ ranges in the family of $d$ tuples of elements of $I_1\oplus \cdots \oplus I_n$ generating an ideal of finite colength. We observed that the integer $\sigma(I_1,\dots,I_n)$, when finite, is equal to the multiplicity defined by Rees in \cite[p.\,181]{Reesllibre}, via the notion of general extension of a local ring (see \cite[p.\,145]{Reesllibre} and \cite{RS}), for certain sets of ideals not necessarily of finite colength. This motivates us to refer
to $\sigma(I_1,\dots, I_n)$ as the {\it Rees' mixed multiplicity of $I_1,\dots, I_n$}.

Let $I_1,\dots, I_n$ be ideals of $R$ for which $\sigma(I_1,\dots, I_n)<\infty$. We define the following number (see \cite[p.\,392]{Bivia2009}):
\begin{equation}\label{laerre}
r(I_1,\dots, I_n)=\min\big\{r\in\Z_{\geq 1}: \sigma(I_1,\dots,
I_n)=e\(I_1+\m^r,\dots, I_n+\m^r\)\big\}.
\end{equation}
Therefore, motivated by \cite[Corollary 3.4]{Bivia2009}, we defined the {\it \L ojasiewicz exponent of $I_1,\dots, I_n$}, denoted by
$\LL_0(I_1,\dots, I_n)$, as
$$
\LL_0(I_1,\dots, I_n)=\min_{s\geqslant 1}\frac{r(I_1^s,\dots, I_n^s)}{s}.
$$

If $I$ denotes an $\m$-primary ideal of $R$, then we set $\LL_0^{(i)}(I)=\LL_0(I,\dots, I, \m,\dots, \m)$, where $I$ is repeated $i$ times and $\m$ is repeated $n-i$ times. As shown in \cite[Proposition 3.9]{BF1}, if $R$ is a regular with infinite residue field $\mathbf k$ and we fix an integer $i\in\{1,\dots, n\}$, the number $\LL_0^{(i)}(I)$ is equal to the {\L}ojasiewicz exponent of the image of $I$ in the quotient ring $R/\langle h_1,\dots, h_{n-i}\rangle$, where $h_1,\dots, h_{n-i}$ are linear forms chosen generically in $\mathbf k[x_1,\dots, x_n]$.
So the numbers $\LL_0^{(i)}(I)$ coincide with the numbers $\nu^{(i)}_I$ defined by Hickel in \cite[p.\,635]{Hickel2}.
As a consequence of \cite[Proposition 3.5]{Bivia2020}, we have that
$$
\LL_0^{(i)}(I)=\inf\left\{\frac rs: r,s\in\Z_{\geq 1}, r\geq s,\,e_{i}(I^s+\m^r)=e_{i}(I^s)\right\}.
$$
for all $i=1,\dots, n$.

Given a subset $\tL\subseteq\{1,\dots,n\}$, $\tL\neq\emptyset$, we define
$\C^n_\tL=\{x\in\C^n: x_i=0,\,\textnormal{for all $i\notin \tL$}\}$. The natural inclusion $\varphi:\C^n_\tL\to \C^n$
induces a morphism $\varphi^*:\O_n\to\O_{n,\tL}$ by composition with $\varphi$, where $\O_{n,\texttt L}$ denote the subring of $\O_n$ formed by all function germs of $\O_n$ that depend only on the variables $x_i$ such that $i\in \texttt L$. Hence, given any ideal $I$ of $\O_n$,
we denote by $I^{\tL}$ the image $\varphi^*(I)$. Let us recall the following result about the computation of the sequence $\LL^*_0(I)$ when $I$ is a monomial ideal of $\O_n$.

\begin{thm}\label{Lojorder}\cite{BF1}
Let us fix a coordinate system $x_1,\dots, x_n$ in $\C^n$. Let $I$ be a monomial ideal of $\O_n$
with respect to this coordinate system and let us suppose that $I$ has finite colength. Then, for all $i\in\{1,\dots, n\}$, we have
\begin{equation}\label{eqLojorder}
\LL_0^{(i)}(I)=\max\big\{\ord(I^{\tL}): \tL\subseteq\{1,\dots, n\},\,\vert \tL\vert=n-i+1\big\}.
\end{equation}
\end{thm}

Here we recall a result proven in \cite{BF1} that we will apply later.

\begin{prop}\cite{BF1}\label{linkmultLs}
Let $(R,\m)$ be a quasi-unmixed Noetherian local ring of dimension $n$.
Let $I_1,\dots, I_n, J$ be ideals of $R$ such that
$\sigma(I_1,\dots, I_n)<\infty$, $\sigma(I_1,\dots, I_{n-1}, J)<\infty$ and $I_n$ has finite colength.
Then
\begin{equation}\label{ineq1}
\frac{\sigma(I_1,\dots, I_n)}{\sigma(I_1,\dots, I_{n-1}, J)}\leq \LL_J(I_n).
\end{equation}

In particular, if $I$ is an ideal of $R$ of finite colength, we have
\begin{equation}\label{ineq2}
\frac{e(I)}{e_{n-1}(I)}\leq \LL_0(I)
\end{equation}
and equality holds if and only if
\begin{equation}\label{ineq22}
e_{n-1}(I)^ne(I)=e(I^{e_{n-1}(I)} + \m^{e(I)}).
\end{equation}
\end{prop}

Let $I$ be an ideal of $\O_n$ of finite colength. By \cite[p.\,635]{Hickel2} (see also \cite[Corollary 3.8]{BF1}) we know that
\begin{equation}\label{ineqHickel}
e(I)\leq \LL_0^{(1)}(I)\cdots \LL_0^{(n)}(I).
\end{equation}
Because of this result, we say that $I$ is a {\it Hickel ideal} or that $I$ has {\it maximal multiplicity}
when equality holds in (\ref{ineqHickel}). We refer to \cite{Bivia2015} for a characterization of
the equality $e(I)=\LL_0^{(1)}(I)\cdots \LL_0^{(n)}(I)$ when $\overline I$ is a monomial ideal.

\section{Log canonical threshold and \L ojasiewicz exponents}\label{lctI}

Let $I$ be an arbitrary $\m$-primary ideal of a given Noetherian local ring $(R,\m)$ of dimension $n$. We define the
{\it Demailly-Pham number of $I$}, denoted by $\DP(I)$, as
\begin{equation}\label{defDP(I)}
\DP(I)=\frac{1}{e_1(I)}+\frac{e_1(I)}{e_2(I)}+\cdots+\frac{e_{n-1}(I)}{e_n(I)}
\end{equation}
where we recall that $e_i(I)$ denotes the mixed multiplicity $e(I,\dots,I,\m,\dots, \m)$, with $I$ repeated $i$ times and $\m$
repeated $n-i$ times, for all $i=1,\dots, n$.

Let $I$ be an ideal of finite colength of $\O_n$. By a result of Demailly and Pham in \cite{DP}, it is known that
\begin{equation}\label{ineqDP}
\DP(I)\leq \lct(I).
\end{equation}
where $\lct(I)$ denotes the log canonical threshold of $I$. If $I$ is an arbitrary ideal of $\O_n$ and $g_1,\dots, g_r$ is a generating system
of $I$, we recall that
\begin{equation}\label{recallLctI}
\lct(I)=\sup\big\{s\in\R_{\geq 0}:
\bigl(|g_1(x)|^{2}+\cdots+|g_r(x)|^2\bigr)^{-s}\ \textrm{ is locally
 integrable at $0$}\big\}.
\end{equation}
It is straightforward to see that the member on the right of (\ref{recallLctI}) does not depend on the
choice of a generating system of $I$. We refer to \cite{Kollar,Mu2} for several equivalent
formulations of the log canonical threshold of an ideal and known properties of this concept.

We will refer to (\ref{ineqDP})
as the Demailly-Pham inequality. The number $\DP(I)$ verifies the following fundamental result (which is equivalent to the Rees' multiplicity theorem).

\begin{thm}\cite{Bivia2017}
Let $R$ be a Noetherian quasi-unmixed local ring.
Let $I_1,I_2$ be two proper ideals of finite colength of $R$ such that $I_1\subseteq I_2$. Then
\begin{equation}\label{esdeDP}
\DP(I_1)\leq\DP(I_2)
\end{equation}
and equality holds if and only if $\overline{I_1}=\overline{I_2}$.
\end{thm}

We say that a given property $(P_t)$ depending on a parameter $t\in\C^r$ holds for all $\vert t\vert \ll 1$ if there exists
an open ball $U$ centered at $0$ in $\C^r$ such that the said property $(P_t)$ holds provided that $t\in U$.

Let $f_t:(\C^n,0)\to (\C,0)$ be an analytic deformation, where $t$ belongs to a parameter space $\C^r$. Let us suppose that $f_t$ has an isolated singularity at the origin, for all $\vert t\vert \ll 1$. Then $\DP(J(f_t))$ is lower semicontinuous, that is, $\DP(J(f_0))\leq\DP(J(f_t))$, for all $\vert t\vert \ll 1$ (see \cite[Corollary 12]{Bivia2017}). Moreover $\DP(J(f_t))$ is constant, for $\vert t\vert \ll 1$, if and only if $\mu^*(f_t)$ is constant, for $\vert t\vert \ll 1$ (see \cite[Corollary 12]{Bivia2017} also).

Let $(R,\m)$ be a Noetherian local ring of dimension $n$. Given an $\m$-primary ideal $I$ of $R$, we define
the number $\Li(I)$ given by
$$
\Li(I)=\sum_{i=1}^n\frac{1}{\LL_0^{(i)}(I)}.
$$

\begin{prop}\label{propdesLDP}
Let $(R,\m)$ be a Noetherian local ring of dimension $d$.
Let $I$ be an $\m$-primary ideal of $R$. Then
\begin{equation}\label{ineqLDP}
\Li(I)\leq \DP(I)
\end{equation}
Moreover, the following conditions are equivalent:
\begin{enumerate}[label=\textnormal{(\alph*)}]
\item\label{propeq} Equality holds in \textnormal{(\ref{ineqLDP})}.
\item\label{propeq2} $\LL_0^{(i)}(I)=\frac{e_i(I)}{e_{i-1}(I)}$, for all $i=1,\dots, n$.
\item\label{propmults} $e_{i-1}(I)^{i}e_i(I)=e_i(I^{e_{i-1}(I)} + \m^{e_i(I)})$, for all $i=1,\dots, n$.
\item\label{propHick} $I$ is a Hickel ideal.
\end{enumerate}
\end{prop}

\begin{proof}
By relation (\ref{ineq2}) we have that item\,\ref{propeq} is equivalent to saying that
$\LL_0^{(i)}(I)=\frac{e_i(I)}{e_{i-1}(I)}$, for all $i=1,\dots, n$. Therefore
the equivalence between \ref{propeq}, \ref{propeq2} and \ref{propmults} follows as a direct application of
Proposition \ref{linkmultLs}. The equivalence between \ref{propeq2} and \ref{propHick} is proven in
\cite[Lemma 5.5]{BF2}.
\end{proof}

\begin{rem}
In general, if $I$ is an ideal of $\O_n$ of finite colength, there is no relation between $B(I)$ and the lower bound $ne(I)^{-\frac 1n}$ for $\lct(I)$ shown by de Fernex et al. in \cite[Theorem 1.2]{dFEM1}. For instance, let us consider the ideal of $\O_3$
given by $I=\langle xy^2z, x^5, y^6, z^5\rangle$.
We have $\Li(I)=\frac{37}{60}$ and $e(I)=110$; consequently $\Li(I)<3e(I)^{-\frac 13}$.
Let us consider the ideal $J=\langle xyz, x^5, y^6, z^5\rangle$. In this case we have $e(J)=85$ and $\Li(J)=\frac{7}{10}$.
Therefore $\Li(J)>3e(J)^{-\frac 13}$.
\end{rem}

Let $I$ and $J$ be ideals of a Noetherian local ring $(R,\m)$ of dimension $n$. We recall that $I$ and $J$ are called {\it projectively equivalent} when there exist integers $a,b\geq 1$ such that $\overline{I^a}=\overline{J^b}$ (see \cite{CHRR} or \cite[p.\,210]{HS}). We remark that $I$ is projectively equivalent to $\m$ if and only if there exists some $r\in\Z_{\geq 1}$ such that $\overline I=\m^r$. Let us recall that if $I$ and $J$ are $\m$-primary, then $e(IJ)^{1/n}\leq e(I)^{1/n}+e(J)^{1/n}$ (see \cite{RSharp}), and, if $R$ is quasi-unmixed, then equality
holds if and only if $I$ and $J$ are projectively equivalent (see \cite{Katz,RSharp} and \cite[p.\,359]{HS}). We refer to \cite{BLQ} for a recent generalization of this result.

\begin{thm}\label{ThmdesLi}
Let $I$ and $J$ be proper ideals of $\O_n$ of finite colength. Then
\begin{equation}\label{desLi}
\LL_0^{(i)}(IJ)\leq \LL_0^{(i)}(I)+\LL_0^{(i)}(J)
\end{equation}
for all $i\in\{1,\dots, n\}$. Equality holds provided that either $I$ and $J$ are projectively equivalent or $J=\m_n$.
\end{thm}

\begin{proof}
Let us prove first that
\begin{equation}\label{LojIJ}
\LL_0(IJ)\leq \LL_0(I)+\LL_0(J).
\end{equation}

Let $a,b,r,s\in \Z_{\geq 1}$ such that $\LL_0(I)\leq \frac ab$ and $\LL_0(J)\leq \frac rs$. Therefore,
$\m^a\subseteq\overline{I^b}$ and $\m^r\subseteq\overline{J^s}$.
Consequently, we have
\begin{align*}
\m^{sa}&\subseteq{\overline{I^b}}^{\,s}\subseteq \overline{{\overline{I^b}}^{\,s}}=\overline{I^{bs}}\\
\m^{rb}&\subseteq\overline{J^s}\subseteq \overline{{\overline{J^s}}^{\,b}}=\overline{J^{bs}}.
\end{align*}

Hence
$$
\m^{sa+rb}\subseteq \overline{I^{bs}}\,\overline{J^{bs}}
\subseteq\overline{{\overline{I^{bs}}\,\overline{J^{bs}}}}=\overline{(IJ)^{bs}}.
$$
The above relation shows that $\LL_0(IJ)\leq \frac{sa+rb}{bs}=\frac ab+\frac rs$. Therefore, relation (\ref{LojIJ}) follows.
Let us remark that the same argument shown above also proves relation (\ref{LojIJ}) for any pair of proper ideals $I$ and $J$ of
any given Noetherian local ring.

Let us fix an integer $i\in\{1,\dots, n-1\}$.  By \cite[Lemma 3.9]{BF1},
there exist a family of $n-i$ linear forms $h_1,\dots, h_{n-i}\in\C[x_1,\dots, x_n]$ such that, if $H$ denotes the ideal of $\O_n$ generated by these forms and $R_H$ denotes
the quotient ring $\O_n/H$, then we have
\begin{align}
\LL_0^{(i)}(I)&=\LL_0(IR_H)   \label{LiH1}\\
\LL_0^{(i)}(J)&=\LL_0(JR_H)   \label{LiH2}\\
\LL_0^{(i)}(IJ)&=\LL_0(IJR_H).\label{LiH3}
\end{align}

Therefore, by applying relation (\ref{LojIJ}) in the context of ideals of $R_H$, we obtain (\ref{desLi}).

Let us suppose that $I$ and $J$ are projectively equivalent. So let $a,b\in \Z_{\geq 1}$ such that
$\overline{I^a}=\overline{J^b}$. In particular we obtain that $a\LL_0(I)=b\LL_0(J)$ and
$$
\overline{I^aJ^a}=\overline{J^bJ^a}=\overline{J^{a+b}}.
$$
Therefore
$$
a\LL_0(IJ)=\LL_0(I^aJ^a)=\LL_0(J^{a+b})=(a+b)\LL_0(J).
$$
That is
$$
\LL_0(IJ)=\frac{a+b}{b}\LL_0(J)=\frac{a}{b}\LL_0(J)+\LL_0(J)=\LL_0(I)+\LL_0(J).
$$

The same argument shown above serves to show the equality $\LL_0(IJ)=\LL_0(I)+\LL_0(J)$ for any pair
of projectively equivalent ideals $I$ and $J$ of
any given Noetherian local ring. Let us fix an index $i\in\{1,\dots,n\}$ and let us consider
relations (\ref{LiH1})-(\ref{LiH3}), for a given ideal $H$ generated by $n-i$ linear forms of $\C[x_1,\dots, x_n]$.
The projective equivalence $I$ and $J$ implies that $IR_H$ and $JR_H$ are projectively equivalent,
since $\overline{I^aR_H}=\overline{\overline{J^b}R_H}$.
So, by relations (\ref{LiH1})-(\ref{LiH3})  we conclude that
$$
\LL_0^{(i)}(IJ)=\LL^{(i)}_0(I)+\LL^{(i)}_0(J)
$$
for all $i=1,\dots, n$.

Let us suppose now that $J=\m_n$. Let $\frac{p}{q}\geq \LL_0(\m_nI)$, where $p,q\in\Z_{\geq 1}$. Since $\m_nI$ is contained in $\m_n$, we have
$\LL_0(\m_nI)\geq \LL_0(\m_n)=1$, so $p\geq q$. Therefore we have
$$
\m_n^q\m_n^{p-q}\subseteq \overline{\m_n^qI^q}.
$$
Thus the cancellation theorem for integral closures (see \cite[p.\,20]{HS}) implies that
$$
\m_n^{p-q}\subseteq \overline{I^q}.
$$
So $\LL_0(I)\leq \frac{p-q}{q}=\frac{p}{q}-1$. In particular, we obtain that $\LL_0(\m_nI)\geq \LL_0(I)+1$.
The reverse inequality holds by (\ref{LojIJ}), thus $\LL_0(\m_nI)=\LL_0(I)+1$.
Let $i\in\{1,\dots, n-1\}$. As before, by considering the images of $I$ and $\m_nI$ in $R_H$, being $H$ an ideal
generated by $n-i$ linear forms of $\C[x_1,\dots, x_n]$, we also conclude that $\LL^{(i)}_0(\m_nI)=\LL^{(i)}_0(I)+1$.
\end{proof}

\begin{rem}
Obviously, the inequality $\LL_0(IJ)\leq \LL_0(I)+\LL_0(J)$ can be strict. This happens, for instance,
for the ideals of $\O_2$ given by $I=\langle x^2, y^3\rangle$ and $J=\langle x^3, y^2\rangle$.
\end{rem}

In the next result we generalize the inequality of \cite[Theorem 1.1]{Kim} to any pair of ideals $I$ and $J$.
We provide a chain of inequalities relating mixed \L ojasiewicz exponents, mixed multiplicities and log canonical thresholds; we also give conditions implying that the corresponding equalities hold.

Let us first recall the main result of Howald in \cite{Howald} stating that if $I$ is a monomial ideal of $\O_n$
(with respect to a given system of coordinates in $\C^n$), then $\lct(I)$ can be obtained through the Newton polyhedron of $I$
by the formula
\begin{equation}\label{mainHowald}
\lct(I)=\frac{1}{\min\{\mu>0: \mu\cdot(1,\dots, 1)\in\Gamma_+(I)\}}.
\end{equation}
As shown in \cite[Theorem 5.4]{BF2}, if $x_1,\dots, x_n$ denote the coordinates in $\C^n$ with respect to which
$I$ is monomial, the right part of (\ref{mainHowald}) can be written as $\frac{1}{\LL_{x_1\cdots x_n}(I)}$.

\begin{thm}\label{novafita} Let $I$ and $J$ be ideals of $\O_n$ of finite colength. Then
\begin{equation}\label{desKim}
\sum_{i=1}^n\frac{1}{\LL_0^{(i)}(I)+\LL_0^{(i)}(J)} \leq \sum_{i=1}^n\frac{1}{\LL_0^{(i)}(IJ)}\leq \DP(IJ)\leq \lct(IJ).
\end{equation}
Moreover the following holds:
\begin{enumerate}[label=\textnormal{(\alph*)}]
\item\label{igs1} If $I$ and $J$ are projectively equivalent and $I$, or $J$, is diagonal, then all inequalities of \textnormal{(\ref{desKim})} become equalities.
\item\label{igs2} If $I$ and $J$ are projectively equivalent or $J=\m_n$, then equality holds in the first inequality of \textnormal{(\ref{desKim})}.
\item\label{igs3} Equality holds in the second inequality of \textnormal{(\ref{desKim})} if and only if $IJ$ is a Hickel ideal.
\end{enumerate}
\end{thm}

\begin{proof}
The first inequality of (\ref{desKim}) is a direct consequence of (\ref{desLi}). The second inequality comes from
Proposition \ref{propdesLDP}.
The third inequality is the Demailly-Pham inequality (\ref{ineqDP}).

Let us prove item \ref{igs1}. Let us suppose that $I$ and $J$ are projectively equivalent. Let us fix coordinates $x_1,\dots, x_n$ in $\C^n$ with respect to which $I$ is diagonal.
So there exist $r_1,\dots, r_n\in \Z_{\geq 1}$ such that
$$
\overline I=\overline{\langle x_1^{r_1},\dots,x_n^{r_n} \rangle}.
$$
By permuting the order of the variables, if necessary, we can assume that $r_1\leq \cdots \leq r_n$. Let us see that $J$ is also diagonal with respect to $x_1,\dots, x_n$.

Let $a,b\in\Z_{\geq 1}$ such that $\overline{I^a}=\overline{J^b}$. Therefore
\begin{equation}\label{Jb}
\overline{J^b}=\overline{\langle x_1^{ar_1},\dots,x_n^{ar_n} \rangle}.
\end{equation}

The ideal on the right of (\ref{Jb}) is a monomial ideal. So, by the theorem of characterization of ideals with monomial integral closure
(see \cite[Theorem 2.11]{BFS}) we conclude that $J^b$ is Newton non-degenerate.
In particular, $J$ is Newton non-degenerate. This implies that $\LL^*_0(J)\in \Z^n_{\geq 1}$, by (\ref{eqLojorder}).
Also by (\ref{eqLojorder}), we have that
$$
\LL_0^{(i)}(J^b)=b\LL_0^{(i)}(J)=ar_i
$$
for all $i=1,\dots, n$. That is $\LL_0^{(i)}(J)=\frac{ar_i}{b}$ and $b$ divides $ar_i$ for all $i=1,\dots, n$.
Let us denote $\LL_0^{(i)}(J)$ by $s_i$, for all $i=1,\dots, n$. Therefore
$$
\overline{J^b}=\overline{\langle x_1^{a r_1},\dots,x_n^{a r_n} \rangle}=
\overline{\langle x_1^{bs_1},\dots,x_n^{bs_n} \rangle}=\overline{\langle x_1^{s_1},\dots,x_n^{s_n} \rangle^b}.
$$
Hence we have
$$
\overline J=\overline{\langle x_1^{s_1},\dots,x_n^{s_n} \rangle}.
$$
That is, $J$ is diagonal. Since $\frac{r_i}{s_i}=\frac ab$, for all $i=1,\dots, n$, we also conclude that
$$
\overline{IJ}=\overline{\overline I\,\overline J}=\overline{\langle x_1^{r_1+s_1},\dots, x_n^{r_n+s_n}\rangle}.
$$
So $IJ$ is also diagonal, which implies that $\LL_0^{(i)}(IJ)=\frac{1}{r_i+s_i}$, for all $i=1,\dots, n$. By Howald's identity (\ref{mainHowald}) we have that
$$
\lct(IJ)=\frac{1}{r_1+s_1}+\cdots+\frac{1}{r_n+s_n}=\Li(IJ).
$$
Hence item \ref{igs1} follows.

Items \ref{igs2} and \ref{igs3} are direct applications of Proposition \ref{propdesLDP} and Theorem \ref{ThmdesLi},
respectively.
\end{proof}

\begin{rem}
\begin{enumerate}[label=(\alph*),wide]\itemsep0.5em
\item Let us fix coordinates $x_1,\dots, x_n$ in $\C^n$. Let $I$ be an ideal of $\O_n$ and let $I^0$ denote the ideal of $\O_n$ generated by all monomials $x_1^{k_1}\cdots x_n^{k_n}$ for which $(k_1,\dots, k_n)\in\Gamma_+(I)$. Obviously $I\subseteq I^0$. In \cite[Theorem 13]{Bivia2017} we proved that if $\lct(I)=\lct(I^0)$,
    then $\DP(I)=\lct(I)$ if and only if $I$ is diagonal. In \cite[Corollary 4.9]{BF2} we proved that if $\lct(I^0)\leq 1$ and there exists some $g\in I$ such that $g$ is Newton non-degenerate and $\Gamma_+(g)=\Gamma_+(I)$, then $\lct(I)=\lct(I^0)$. The condition $\lct(I^0)\leq 1$ can be checked with the aid of Howald's identity (\ref{mainHowald}).
The characterization of the equality $\DP(I)=\lct(I)$ is still an open problem.
We conjecture that, given an ideal $I$ of $\O_n$, the equality $\DP(I)=\lct(I)$ holds if and only if
$I$ becomes diagonal with respect to a suitable coordinate system $x_1,\dots, x_n$, that is, there
exists a biholomorphism $\varphi:(\C^n,0)\to (\C^n,0)$ and coordinates $x_1,\dots, x_n$ of $\C^n$ in the domain of $\varphi$ such that $\varphi^*(I)$ is diagonal with respect to $x_1,\dots, x_n$.

\item Let us suppose that $I$ and $J$ are projectively equivalent ideals, so let $a,b\in\Z_{\geq 1}$ such that
$\overline{I^a}=\overline{J^b}$. Then it follows that $\overline{I^bJ^b}=\overline{I^{a+b}}$ and therefore
\begin{align*}
\Li(IJ)&=\frac{b}{a+b}\Li(I)\\
\DP(IJ)&=\frac{b}{a+b}\DP(I)\\
\lct(IJ)&=\frac{b}{a+b}\lct(I).
\end{align*}
Consequently the condition $\Li(I)=\DP(I)$ holds if and only if $\Li(J)=\DP(J)$ for any, or for some, ideal $J$ projectively equivalent to $I$. An analogous assertion holds for the condition $\DP(I)=\lct(I)$.

\item The inequalities of (\ref{desKim}) extend naturally to any product of ideals. In particular, if $I_1,\dots, I_r$ are ideals of
$\O_n$ of finite colength, then
$$
\sum_{i=1}^r\frac{1}{\LL_0^{(i)}(I_1)+\cdots+\LL_0^{(i)}(I_r)}\leq \lct(I_1\cdots I_r).
$$

\end{enumerate}
\end{rem}

In the next result we analyze the case $J=\m_n$ of Theorem \ref{novafita}.

\begin{thm}\label{novafita2} Let $I$ be an ideals of $\O_n$ of finite colength. Then
\begin{equation}\label{desKim2}
\sum_{i=1}^n\frac{1}{\LL_0^{(i)}(I)+1} = \Li(\m_nI)\leq \DP(\m_nI)\leq \lct\(\m_nI\).
\end{equation}
Moreover, all inequalities of \textnormal{(\ref{desKim2})} become equalities if and only if
there exists some $r\in\Z_{\geq 1}$ such that $\overline I=\m_n^r$.
\end{thm}

\begin{proof}
Relation (\ref{desKim2}) is the case $J=\m_n$ of (\ref{desKim}), joined with the fact that
$\LL^{(i)}_0(\m_nI)=\LL^{(i)}_0(I)+1$, for all $i=1,\dots, n$ (Theorem \ref{ThmdesLi}). Let us suppose that $\overline I=\m_n^r$, for some $r\in\Z_{\geq 1}$. In particular, $r=\ord(I)$. Hence $\overline{\m_nI}=\m_n^{r+1}$ and this implies that $\LL_0^{(i)}(\m_nI)=\LL_0^{(i)}(\m_n^{r+1})=r+1$ and
$\lct(\m_nI)=\lct(\m_n^{r+1})=\frac{1}{r+1}$. Therefore the inequalities of \textnormal{(\ref{desKim2})} become equalities.

Conversely, let us suppose that the inequalities of \textnormal{(\ref{desKim2})} are equalities. In particular
we have $\Li(\m_nI)=\DP(\m_nI)$. This implies that $\m_nI$ is a Hickel ideal, by Proposition \ref{propdesLDP}.

Let $a_i=\LL_0^{(i)}(I)$, for all $i=1,\dots, n$. Therefore $a_1\leq \cdots \leq a_n$ and we have $\LL_0^{(i)}(\m_nI)=a_i+1$, for all $i=1,\dots, n$.
Let us first observe that
\begin{align}
e(\m_nI)&=\sum_{i=0}^n \binom{n}{i}e_i(I)\leq 1+\sum_{i=1}^n \binom{n}{i}a_1\cdots a_i \label{comb0}   \\
\LL_0^{(1)}(\m_nI)\cdots \LL_0^{(n)}(\m_nI)&=(a_1+1)\cdots (a_n+1)=1+\sum_{i=1}^n\(\sum_{1\leq k_1<\cdots <k_i\leq n}a_{k_1}\cdots a_{k_i}\).
\label{comb1}
\end{align}
where the first equality of (\ref{comb0}) comes from the formula for the multiplicity of a product of ideals (see \cite[p.\,490]{Rees2}).

The property $a_1\leq \cdots \leq a_n$ implies that, given any index $i\in\{1,\dots, n\}$, we have
\begin{equation}\label{comb3}
a_1\cdots  a_i\leq a_{k_1}\cdots a_{k_i}
\end{equation}
for all set of indices $\{k_1,\dots ,k_i\}\subseteq\{1,\dots,n\}$
for which $1\leq k_1<\cdots <k_i\leq n$. This fact implies that
\begin{equation}\label{comb2}
e(\m_nI)\leq 1+\sum_{i=1}^n \binom{n}{i}a_1\cdots a_i\leq 1+\sum_{i=1}^n\(\sum_{1\leq k_1<\cdots <k_i\leq n}a_{k_1}\cdots a_{k_i}\).
\end{equation}
Since $\m_nI$ is Hickel, we have $e(\m_nI)=\LL_0^{(1)}(\m_nI)\cdots \LL_0^{(n)}(\m_nI)$, and hence relations (\ref{comb1}) and (\ref{comb2}) imply that $a_1\cdots  a_i=a_{k_1}\cdots a_{k_i}$, for all $1\leq k_1<\cdots <k_i\leq n$ and all $i\in\{1,\dots, n\}$.
In particular $a_1=\cdots=a_n=\ord(I)$. Hence
$$
e(I)\leq a_1\cdots a_n=(\ord(I))^n.
$$
Moreover, the inclusion $I\subseteq \m_n^{\ord(I)}$ implies that $e(I)\geq e(\m_n^{\ord(I)})=(\ord(I))^n$. Hence $e(I)=(\ord(I))^n$
and therefore $\overline I=m_n^{\ord(I)}$, by the Rees' multiplicity theorem.

\end{proof}

\begin{rem}
Let $I$ be an ideal of $\O_n$ of finite colength and let $g_1,\dots, g_m$ be any generating system of $I$. Let $r=\ord(I)$.
We remark that, by the characterization of the ideals with monomial integral closure \cite[Theorem 2.11]{BFS}, the condition $\overline I=\m_n^r$ holds if and only if the ideal of $\O_n$
generated by $j^rg_1,\dots, j^rg_m$ has finite colength, where $j^rg$ stands for the $r$-jet of a given function $g\in\O_n$
(that is, the sum of all terms in the Taylor expansion of $g$ of at most degree $r$).
\end{rem}

\begin{cor}\label{novafita3}
Let $f:(\C^n,0)\to (\C,0)$ be an analytic function germ with isolated singularity at the origin. Then
\begin{equation}\label{desKim3}
\sum_{i=1}^n\frac{1}{\LL_0^{(i)}(J(f))+1}= \Li\big(\m_nJ(f)\big)\leq \DP\big(\m_nJ(f)\big)\leq \lct\big(\m_n J(f)\big)
\end{equation}
Moreover, the following conditions are equivalent:
\begin{enumerate}[label=\textnormal{(\alph*)}]
\item\label{novafita3a} The inequalities of \textnormal{(\ref{desKim3})} become equalities.
\item\label{novafita3b} $\overline{J(f)}=\m_n^{\ord(f)-1}$.
\item\label{novafita3c} The polynomial $j^rf$ has an isolated singularity at the origin, where $r=\ord(f)$.
\end{enumerate}
\end{cor}

\begin{proof}
Relation (\ref{desKim3}) is a direct application Theorem \ref{novafita} by considering the case $I=\m_n$ and $J=J(f)$.
Since $\ord(J(f))=\ord(f)-1$, the equivalence of \ref{novafita3a} and \ref{novafita3b} follows from Theorem \ref{novafita2}.

Let us see the equivalence between \ref{novafita3b} and \ref{novafita3c}. Let $g=j^rf$, where $r=\ord(f)$, and let $\Delta$
denote the unique compact face of $\Gamma_+(\m_n^{r-1})$ of dimension $n-1$. Let $f=\sum_ka_kx^k$ be the Taylor expansion of $f$ around the
origin. For any given
$k\in \R^n$, we denote by $\vert k\vert$ denote the sum of the components of $k$. Given an index $i\in\{1,\dots, n\}$, if $e_1,\dots, e_n$ denotes the canonical basis of $\R^n$, then we have
\begin{align*}
\supp\(\frac{\partial g}{\partial x_i}\)&=\big\{k-e_i: \vert k\vert=r,\, k_i>0,\, a_k\neq 0\big\}\\
\supp\(\frac{\partial f}{\partial x_i}\) \cap \Delta&=\big\{k-e_i: \vert k-e_i\vert=r-1,\, k_i>0,\, a_k\neq 0\big\}.
\end{align*}
Therefore we see that
\begin{equation}\label{partprincipal}
\frac{\partial g}{\partial x_i}=\(\frac{\partial f}{\partial x_i}\)_{\!\!{\Delta}}
\end{equation}
for all $i=1,\dots, n$.

By \cite[Theorem 2.11]{BFS}, where the ideals with monomial integral closure are characterized, condition \ref{novafita3b} is equivalent to saying that the ideal of $\O_n$ generated by
$$
{\(\frac{\partial f}{\partial x_1}\)}_{\!\!\Delta},\dots,{\(\frac{\partial f}{\partial x_n}\)}_{\!\!{\Delta}}
$$
has finite colength in $\O_n$. Therefore, relation (\ref{partprincipal}) shows the equivalence between
items \ref{novafita3b} and \ref{novafita3c}.
\end{proof}

We also refer to \cite[Proposition 2.7]{Cargese} for further characterizations of the condition $\overline{J(f)}=\m_n^{\ord(f)-1}$.
In the next example we illustrate the sharpness of the chain of inequalities given in (\ref{desKim3}).

\begin{ex}\label{BS}
Let us consider the analytic family of function germs $f_t:(\C^3,0)\to (\C,0)$ given by the Brian\c con-Speder example (see \cite{BSexample}). That is
$$
f_t(x,y,z)=x^5+z^{15}+y^7z+t xy^6
$$
for all $(x,y,z)\in\C^3$, $t\in\C$. We recall that $f_t$ is weighted homogeneous with respect to the vector of weights
$w=(3,2,1)$, for all $\vert t\vert\ll 1$. As explained in \cite[p.\,805]{BF2}, we have that
\begin{equation}\label{exemplesstars}
\mu^*(f_t)=
\begin{cases}
(364, 28, 4)  &\textrm{if $t=0$}\\
(364, 26, 4)  &\textrm{if $t\neq 0$}
\end{cases}
\hspace{1cm}
\LL^*_0(J(f_t))=
\begin{cases}
(14,7,4)  &\textrm{if $t=0$}\\
(14,6.5,4)  &\textrm{if $t\neq 0$}.
\end{cases}
\end{equation}
Therefore we see that $J(f_t)$ is a Hickel ideal if and only if $t\neq 0$.

Using (\ref{exemplesstars}) it follows that
$$
\Li(\m_3J(f_0))=\frac{1}{\LL_0^{(1)}(J(f_0))+1}+\frac{1}{\LL_0^{(2)}(J(f_0))+1}+\frac{1}{\LL_0^{(3)}(J(f_0))+1}=\frac{47}{120}\simeq 0.391
$$
Moreover
\begin{align*}
\DP(\m_nJ(f_0))&=\frac{1}{e_1(\m_nJ(f_0))}+\frac{e_1(\m_nJ(f_0))}{e_2(\m_nJ(f_0))}+\frac{e_2(\m_nJ(f_0))}{e_3(\m_nJ(f_0))}\\
&=
\frac{1}{5}+\frac{5}{37}+\frac{37}{461}=\frac{35437}{85285}\simeq 0.415
\end{align*}

Since the ideal $J(f_0)$ is Newton non-degenerate, then $\m_3J(f_0)$ is also. Let us also remark that
$$
\overline{\m_3J(f_0)}=\overline{\m_3\big\langle x^4, y^7, z^{14}, y^6z\big\rangle}.
$$

Therefore, by applying (\ref{mainHowald}), we have
$$
\lct(\m_3J(f_0))=\lct\(z^{15}, yz^{14}, xz^{14}, y^6z^2, y^7z, xy^6z, y^8, xy^7, x^4z, x^4y, x^5\)=\frac{4}{9}\simeq 0.444
$$

Let us fix now a small enough parameter $t\neq 0$. Also by (\ref{exemplesstars}), we obtain
$$
\Li(\m_3J(f_t))=\frac{1}{\LL_0^{(1)}(J(f_t))+1}+\frac{1}{\LL_0^{(2)}(J(f_t))+1}+\frac{1}{\LL_0^{(3)}(J(f_t))+1}=\frac{2}{5}=0.4
$$
We also deduce that
\begin{align*}
\DP(\m_nJ(f_t))&=\frac{1}{e_1(\m_nJ(f_t))}+\frac{e_1(\m_nJ(f_t))}{e_2(\m_nJ(f_t))}+\frac{e_2(\m_nJ(f_t))}{e_3(\m_nJ(f_t))}\\
&=\frac{1}{5}+\frac{5}{35}+\frac{35}{455}=\frac{191}{455}\simeq 0.419
\end{align*}

Let $J$ denote the ideal of $\O_3$ generated by the monomials $x^{k_1}y^{k_2}z^{k_3}$ whose support $(k_1,k_2,k_3)$ belongs to $\Gamma_+(\m_3J(f_t))$. Therefore, we have
$$
J=\overline{\m_3\big\langle x^4, y^6, z^{14}\big\rangle}.
$$
By applying (\ref{mainHowald}), we obtain that
$\lct(J)=\frac{41}{90}$.
Let us consider the function $g$ given by
$$
g=x\frac{\partial f_t}{\partial x}+y\frac{\partial f_t}{\partial y}+z\frac{\partial f_t}{\partial z}.
$$
Obviously $g\in \m_3J(f_t)$. By using Singular \cite{Singular}, is straightforward to see that $\Gamma_+(g)=\Gamma_+(J)$ and $g$ is Newton non-degenerate. Hence by \cite[Corollary 4.9]{BF2}, we conclude the equality
$$
\lct(\m_3J(f_t))=\lct(J)=\frac{41}{90}\simeq 0.455.
$$

The computation of $\lct(\m_3J(f_0))$ and $\lct(J)$ via Howald's identity (\ref{mainHowald})
has been carried on with the help of the program {\it G\'ermenes} developed by A.\,Montesinos-Amilibia \cite{Montesinos}.
All the mixed multiplicities involved in the above computations have been obtained with the aid of Singular \cite{Singular}.
\end{ex}


\end{document}